\documentclass[12pt]{amsart}

\usepackage{amsmath,color,graphicx,amsthm,amssymb, mathtools, enumerate, pgfplots, bm, float, hyperref}
\pgfplotsset{compat=1.16}

\newtheorem{theorem}{Theorem}[section]

\newtheorem{proposition}[theorem]{Proposition}

\newtheorem{lemma}[theorem]{Lemma}

\newtheorem{definition}[theorem]{Definition}

\newtheorem*{theorem*}{Theorem}

\newcommand{\ilim}{\varprojlim}

\newtheorem{example}[theorem]{Example}

\numberwithin{equation}{section}

\newcommand{\invlim}{\underleftarrow{\lim}}

\begin{document}

\title[Periodicity and Indecomposability]{Periodicity and Indecomposability in Generalized Inverse Limits}

\author{Tavish J.~Dunn and David J.~Ryden}

\address{Baylor University \\ Waco, Texas 76798-7328}

\email{tavish\_dunn@baylor.edu, david\_ryden@baylor.edu}

\subjclass[2020]{}

\keywords{continuum, inverse limit, periodicity, indecomposability, intermediate value property}

\begin{abstract}
In this paper, we consider inverse limits of $[0,1]$ using upper semicontinuous set-valued bonding functions with the intermediate value property. Expanding on classical results by Barge and Martin, we explore the relationship between periodicity in the bonding function and the topology of the corresponding inverse limit. In particular, for an inverse limit of a single upper semicontinuous bonding map with the intermediate value property, we provide sufficient conditions for the existence of a periodic cycle with period not a power of two in the bonding function to imply the existence of an indecomposable subcontinuum of the inverse limit.  We also give a partial converse.  Along the way to these results, we show that subcontinua of these inverse limits have the full-projection property.

\end{abstract}

\maketitle

\section{Introduction}

Inverse limits have been a connection point between the studies of dynamics and continuum theory for decades.  We note in particular the study undertaken by Marcy Barge and Joe Martin in the eighties to discern for a map $f:[0,1] \to [0,1]$ the relationship of its dynamics to the dynamics of the shift map on its inverse limit \cite{BargeMartin1}, \cite{BargeMartin2} and to the topology of its inverse limit \cite{BargeMartin3}, \cite{BargeMartin4}.  Two of their results \cite[Theorems~1 \& 7]{BargeMartin1} are of primary concern here: (1) If $f$ has a periodic point whose period $p$ is not a power of two, then the inverse limit has an indecomposable subcontinuum that is invariant under the shift $\hat{f}^p$, and (2) If $f$ is organic and the inverse limit is indecomposable, then $f$ has a periodic cycle whose period is not a power of two.

The introduction by William S.~Mahavier and W.T.~Ingram of inverse limits with upper semicontinuous set-valued bonding functions \cite{Mahavier}, \cite{IngramMahavier} induced a spate of questions about what conditions allow for the generalization of classical results to this new context.  A particular area of interest has been to identify and analyze circumstances that give rise to indecomposable subcontinua in the inverse limit. Ingram, James P.~Kelly, Jonathan Meddaugh, and Scott Varagona have all written on the subject, using the full-projection property as a crucial tool to demonstrate indecomposability \cite{Ingram}, \cite{Kelly}, \cite{KellyMeddaugh}, \cite{Varagona}.

Similarly, much work has been done in set-valued dynamical systems in an attempt to generalize results from the classical case, with applications to switched circuit networks \cite{Camlibel}, economics \cite{Cherene}, and game theory \cite{MaschlerPeleg} \cite{FaureRoth}. Recent years have seen the field as a robust area of study in its own right, with results involving the specification property \cite{Tennant}, chaos \cite{Fedeli}, entropy \cite{Olivera} \cite{KellyTennant}, and shadowing \cite{Pilyugin1} \cite{Pilyugin2}. 

This paper is part of a study that explores the intermediate value property for upper semicontinuous set-valued interval maps.  The intermediate value property guarantees that the Sarkovkii order for periodic cycles holds \cite{OteyRyden}, and the weak intermediate value property gives rise to connectedness of the inverse limit \cite{Dunn}.  In this paper we generalize the two results of Barge and Martin stated above.  In our approach the intermediate value property plays a crucial role in linking the nontrivial dynamics of the bonding function to the exotic topology of the inverse limit.  A key aspect of that role is to establish the full-projection property.  The main result, stated here, follows from Theorems~\ref{Interior} and \ref{IndecomposablePeriodic}.

\begin{theorem}  \label{MainResult}
Suppose $f:[0,1]\to 2^{[0,1]}$ is upper semicontinuous and has the intermediate value property. Consider the following conditions.

\begin{enumerate}[(a)]
\item \label{1.1.a} $f$ has a periodic cycle whose period is not a power of two.
\item \label{1.1.b}
    \begin{enumerate} [(b1)]
    \item \label{1.1.b.1} $\ilim\{[0,1], f\}$ has an indecomposable subcontinuum.
    \item \label{1.1.b.2} $\ilim\{[0,1], f\}$ is indecomposable.
    \end{enumerate}
\end{enumerate}
Then (\ref{1.1.a}) and (\ref{1.1.b}) are related as follows.
\begin{enumerate}
\item \label{MainResult.1}If $f|_{[0,1] \setminus\pi_1[{\rm int}(G(f))]}$ is almost nonfissile and light, then (\ref{1.1.a}) implies (\ref{1.1.b.1}).
\item \label{MainResult.2}If $f$ is organic, then (\ref{1.1.b.2}) implies (\ref{1.1.a}).
\end{enumerate}

\end{theorem}

Section~2 presents preliminary definitions and notation along with some examples of upper semicontinuous functions to illustrate the intermediate value property and weak intermediate value property and note that inverse limits of upper semicontinuous functions with the intermediate value property are connected \cite{Dunn}.

In Section~3, we establish the full-projection property for inverse limits of surjective, light, almost nonfissile, upper semicontinuous functions with the intermediate value property (Theorem~\ref{FPP}) and for all subcontinua whose projections are nondegenerate (Theorem~\ref{FPP2}).

Section~4 generalizes results of Barge and Martin \cite{BargeMartin1} and culminates in Theorems~\ref{Interior} and \ref{IndecomposablePeriodic}, in which the existence of periodic cycles with period not a power of two in the bonding function gives rise to indecomposability in the inverse limit and vice versa, respectively.

\section{Definitions and Notation}

\begin{definition}
\rm{A \emph{continuum} refers to a nonempty, compact, connected subset of a metric space. For a continuum $X$, we denote the collection of nonempty compact subsets of $X$ by $2^{X}$ and denote the collection of nonempty subcontinua of $X$ by $C(X)$. A continuum $X$ is \emph{irreducible} about a closed set $A \subseteq X$ if the only subcontinuum of $X$ containing $A$ is $X$ itself.}
\end{definition}

\begin{definition}
\rm{A function $f:[a,b]\rightarrow2^{[c,d]}$ is \emph{upper semicontinuous at $x$} if for every open set $U$ containing $f(x)$ there is an open set $V$ containing $x$ such that $f[V] \subseteq U$. The function $f$ is \emph{upper semicontinuous} if it is upper semicontinuous at every point in its domain.}

The \emph{graph} of a function $f:[a,b]\rightarrow2^{[c,d]}$ is the set $G(f)=\{(x,y)\in[a,b]\times[c,d]:y\in f(x)\}$. It is well known from \cite{IngramMahavier} that $f$ is upper semicontinuous if and only if $G(f)$ is closed.
\end{definition}

\begin{definition}
\rm{Let $X$ and $Y$ be metric spaces and $f:X\rightarrow 2^{Y}$. An \emph{orbit} of $f$ is a sequence $\{x_{i}\}_{i\in\omega}$ where $x_{i+1}\in f(x_{i})$ for all $i$. If $x\in X$, an orbit of $x$ is an orbit of $f$ where $x_{0}=x$. The orbit is said to be \emph{periodic} if there is some $n\in\mathbb{N}$ such that $x_{n+i}=x_{i}$ for all $i$. The smallest such $n$ is called the \emph{period} of the orbit. A finite sequence $(x_{0},x_{1},\dots,x_{n-1})$ is called a \emph{cycle} if $(x_{0},x_{1},\dots,x_{n-1},x_{0},x_{1},\dots)$ is a periodic orbit.}
\end{definition}

\rm{As $f$ is a set-valued function, a given point $x$ may not have a unique orbit. Because of this, for a given orbit $\{x_{i}\}_{i\in\omega}$ there may be some $i\in\mathbb{N}$ such that $x_{i}=x$ but $\{x_{i}\}_{i\in\omega}$ is not periodic. Similarly there may be an orbit of period $n$, yet there may be some $0<j<n$ such that $x_{j}=x$. This may occur in orbits where $f(x)$ is nondegenerate and $x_{1}\neq x_{j+1}$.}

\begin{definition}
\rm{Let $X_{0},X_{1},X_{2},\dots$ be a sequence of continua and for all $i\in\omega$ let $f_{i+1}:X_{i+1}\rightarrow 2^{X_{i}}$ be upper semicontinuous.  The pair $\{X_{i},f_{i}\}$ is called an \emph{inverse sequence}, and the \emph{inverse limit} of $\{X_{i},f_{i}\}$ is the subspace of $\prod_{i\in\omega}$ given by
\[\invlim\{X_{i},f_{i}\}=\left\{x=(x_{0},x_{1},\dots)\in\prod_{i\in\omega}X_{i}:x_{i-1}\in f_{i}(x_{i}) \ \forall i\geq1\right\}.
\]
The spaces $X_{i}$ are called the \emph{factor spaces} of the inverse limit, and the functions $f_{i}$ are the \emph{bonding functions}.  For $n > i$, $f_i^n: X_n \to X_i$ denotes the composition $f_{i+1} \circ f_{i+2} \circ ... \circ f_n$.  For each $n\in\omega$, the map $\pi_{n}:\invlim\{X_{i},f_{i}\}\rightarrow X_{n}$ defined by $\pi_{n}(x)=x_{n}$ is the \emph{projection map} onto the $n$th factor space.}
\end{definition}

\begin{definition}
\rm{Let $\{X_i,f_i\}$ be an inverse sequence and $X = \invlim \{X_i,f_i\}$.  We say $X$ has the \emph{full-projection property} if and only if $K = X$ for every subcontinuum $K$ of $X$ such that $\pi_i[K] = X_i$ for infinitely many $i \in \omega$.}
\end{definition}

\begin{definition}
\rm{The function $f:[a,b]\rightarrow2^{[c,d]}$ is \emph{weakly continuous from the left at $x$} if it is upper semicontinuous and, for each $y\in f(x)$, there is a sequence $\{(x_{n},y_{n})\}_{n\in\omega}$ that converges to $(x,y)$ such that $x_{n}<x$ and $y_{n}\in f(x_{n})$ for each $n$.

The function $f:[a,b]\rightarrow2^{[c,d]}$ is \emph{weakly continuous from the right at $x$} if it is upper semicontinuous and, for each $y\in f(x)$, there is a sequence $\{(x_{n},y_{n})\}_{n\in\omega}$ that converges to $(x,y)$ such that $x_{n}>x$ and $y_{n}\in f(x_{n})$ for each $n$.

The function $f:[a,b]\rightarrow2^{[c,d]}$ is \emph{weakly continuous at $x$} if $f$ is weakly continuous from the left and from the right at $x$, and $f$ is \emph{weakly continuous} if it is weakly continuous for each $x\in(a,b)$.}
\end{definition}

\begin{definition}
\rm{Let $f:[a,b]\rightarrow 2^{[c,d]}$ be an upper semicontinuous function. We say $f$ has the \emph{intermediate value property} if, given distinct $x_{1},x_{2}$, distinct $y_{1}\in f(x_{1})$, $y_{2}\in f(x_{2})$, and $y$ strictly between $y_1$ and $y_2$, there is some $x$ strictly between $x_{1}$ and $x_{2}$ such that $y\in f(x)$.

We say $f$ has the \emph{weak intermediate value property} if, given distinct $x_{1}$, $x_{2}$, and $y_{1}\in f(x_{1})$ there is some $y_{2}\in f(x_{2})$ such that if $y$ is between $y_{1}$ and $y_{2}$, then there is $x$ between $x_{1}$ and $x_{2}$ such that $y\in f(x)$.}
\end{definition}

Note that we do not specify if $x_{2}$ is larger than $x_{1}$. So for a function to have the weak intermediate value property, it is necessary for the condition to hold when $x_{2}>x_{1}$ and $x_{1}>x_{2}$. If $f$ is upper semicontinuous and has the intermediate value property, it follows that $f$ is weakly continuous via Theorem 3.12 of \cite{Dunn}.

Let $f:[a,b]\rightarrow 2^{[c,d]}$ and $g:[c,d]\rightarrow 2^{[i,j]}$ be upper semicontinuous, $I$ be a closed subinterval of $[a,b]$, and $J$ be a closed subinterval of $[c,d]$ such that if $x\in I$, then $f(x)\cap J\neq\emptyset$. Let $f|_{I}:I\rightarrow2^{[c,d]}$ denote the function $f|_{I}(x)=f(x)$. Let $f|_{I}^{J}:I\rightarrow J$ denote the function $f|_{I}^{J}(x)=f(x)\cap J$. Note that if $f$ and $g$ have the (weak) intermediate value property, then each of $f|_{I}$, $f|_{I}^{J}$, and $g\circ f$ has the (weak) intermediate value property as well.

Below we present some examples to demonstrate what it means for an upper semicontinuous function to have the intermediate value property and weak intermediate value property. Examples \ref{IVPFunction}, \ref{WeakIVP}, and \ref{NoWeakIVP} are upper semicontinuous functions that have the intermediate value property, the weak intermediate value property but not the intermediate value property, and neither property respectively.

\begin{example}\label{IVPFunction}
Let $f:[0,1]\rightarrow 2^{[0,1]}$ be defined by \[f(x)=\left\{\begin{array}{lr}
\left[0,\frac{1}{4}\right] & x=0 \\
\frac{1}{4}\sin\left(\frac{1}{x}\right)+\frac{1}{4} & 0\leq x\leq\frac{1}{\pi} \\
\frac{1}{4(\pi-1)}(3\pi x+\pi-4) & \frac{1}{\pi}\leq x\leq 1
\end{array}\right.\]
Then $f$ has the intermediate value property.
\end{example}

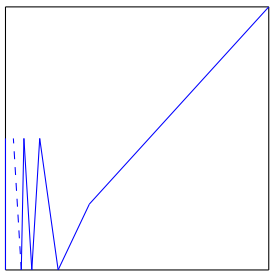
\begin{figure}[H]
\begin{center}
\begin{tikzpicture}[x=3.5cm, y=3.5cm]
\draw (0,1) -- (1,1);
\draw (1,0) -- (1,1);
\draw (0,0) -- (1,0);
\draw (0,0) -- (0,1);
\draw[blue] (0,0) -- (0,1/2);
\draw[blue,domain=1/3.14:1,samples=1000,dashed] (.06,0) -- (.03,1/2);
\draw[blue,domain=1/3.14:1,samples=1000,smooth] (7/100,1/2) -- (.06,0);
\draw[blue,domain=1/3.14:1,samples=1000,smooth] (1/10,0) -- (7/100,1/2);
\draw[blue,domain=1/3.14:1,samples=1000,smooth] (13/100,1/2) -- (1/10,0);
\draw[blue,domain=1/3.14:1,samples=1000,smooth] (1/5,0) -- (13/100,1/2);
\draw[blue,domain=1/3.14:1,samples=1000,smooth] (1/3.14,1/4) -- (1/5,0);
\draw[blue,domain=1/3.14:1,samples=1000,smooth] plot (\x, {1.1*\x-0.1});
\end{tikzpicture}
\end{center}
 \caption{$G(f)$ from Example \ref{IVPFunction}.} \label{Fig1}
\end{figure}

\begin{example}\label{WeakIVP}
Let $f:[0,1]\rightarrow2^{[0,1]}$ be given by $f(x)=\{x,1-x\}$. Then $f$ is upper semicontinuous and has the weak intermediate value property but does not have the intermediate value property.
\end{example}
To see why $f$ does not have the intermediate value property, consider $(x_{1},y_{1})=(0,1)$ and $(x_{2},y_{2})=\left(\frac{1}{4},\frac{1}{4}\right)$. There is no $x\in\left(0,\frac{1}{4}\right)$ such that $f(x)$ contains $\frac{1}{2}$.

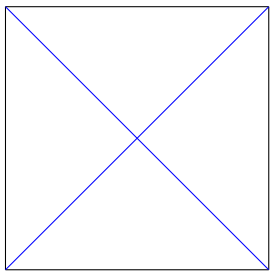
\begin{figure}[h]
\begin{center}
 \begin{tikzpicture}[x=3.5cm, y=3.5cm]
\draw (0,1) -- (1,1);
\draw (1,0) -- (1,1);
\draw (0,0) -- (1,0);
\draw (0,0) -- (0,1);
\draw[blue,domain=0:1,samples=1000] plot (\x, {\x});
\draw[blue,domain=0:1,samples=1000] plot (\x, {1-\x});
\end{tikzpicture}
 \end{center}
 \caption{$G(f)$ from Example \ref{WeakIVP}.} \label{Fig2}
 \end{figure}

\begin{example}\label{NoWeakIVP}
Let $f:[0,1]\rightarrow2^{[0,1]}$ be given by \[f(x)=\left\{\begin{array}{lr}
\frac{1}{3}x & 0\leq x<\frac{1}{2} \\
\left\{\frac{1}{3}x, 2x-1\right\} & \frac{1}{2}\leq x\leq 1
\end{array}\right.\]
Then $f$ has neither the intermediate value property nor the weak intermediate value property.
\end{example}
Let $(x_{1},y_{1})=\left(\frac{1}{2},0\right)$ and $x_{2}=\frac{1}{4}$. Since $f\left(\frac{1}{4}\right)=\left\{\frac{1}{12}\right\}$, $y_{2}$ must be $\frac{1}{12}$. But $\frac{1}{24}\notin f(x)$ for any $x\in\left(\frac{1}{4},\frac{1}{2}\right)$. So $f$ does not have the weak intermediate value property.

\begin{figure}[H]
\begin{center}
 \begin{tikzpicture}[x=3.5cm, y=3.5cm]
\draw (0,1) -- (1,1);
\draw (1,0) -- (1,1);
\draw (0,0) -- (1,0);
\draw (0,0) -- (0,1);
\draw[blue,domain=0:1,samples=1000] plot (\x, {0.5*\x});
\draw[blue,domain=0.5:1,samples=1000] plot (\x, {2*\x-1});
\end{tikzpicture}
 \end{center}
  \caption{$G(f)$ from Example \ref{NoWeakIVP}.} \label{Fig3}
  \end{figure}
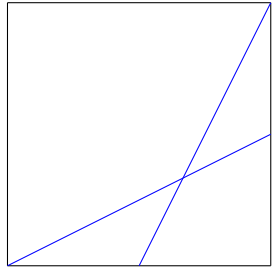


\begin{definition}
\rm{Let $X$ and $Y$ be metric spaces and $f:X\rightarrow 2^{Y}$. A point $x \in X$ is a \emph{fissile} point of $f$ if $|f(x)|>1$ and a \emph{nonfissile} point otherwise, i.e. $f(x)=\{y\}$.

A point $(x,y)\in G(f)$ is a \emph{fissile} point of $G(f)$ if $x$ is a fissile point of $f$ and a \emph{nonfissile} point of $G(f)$ otherwise.

The function $f$ is \emph{almost nonfissile} if the set of nonfissile points of $G(f)$ is a dense $G_{\delta}$ subset of $G(f)$.

Let $\{X_i, f_i\}$ be an inverse sequence. A point $(x_0, x_1,...) \in \invlim\{X_{i},f_{i}\}$ is a fissile point of $\invlim\{X_{i},f_{i}\}$ if $x_i$ is a \emph{fissile} point of $f_i$ for some $i$ and a \emph{nonfissile} point of $\invlim\{X_{i},f_{i}\}$ otherwise.}
\end{definition}


\noindent {\bf Remark.}
The requirement that a function $f:X \to 2^Y$ be almost nonfissile is not equivalent to the requirement that the set of nonfissile points of $f$ be a dense $G_{\delta}$ subset of $X$.  See Examples~\ref{NoIVPNotANF} and \ref{IVPNotANF}. 

However, it is true that if $f$ is almost nonfissile, then the set of nonfissile points of $f$ is a dense $G_{\delta}$ set in $X$.  It is straightforward to show density, and it is shown in Lemma~\ref{ANFDomain} that the set of nonfissile points of $f$ is a $G_{\delta}$ set.

The set of fissile points of $f:X \to 2^Y$, the set of fissile points of $G(f)$, and the set of fissile points of an inverse limit are all $F_{\sigma}$ sets.  The first and last of these is proved in \cite{Ryden}.  The first is also a consequence of Lemma~\ref{ANFDomain}.

\begin{example}  \label{NoIVPNotANF}
Let $f:[0,1] \rightarrow 2^{[0,1]}$ be defined by \[f(x)=\left\{\begin{array}{lr}
\{0\} & 0 \leq x < 1 \\
{[}0,1{]} & x = 1  \\
\end{array}\right.\]
Then $f$ is not almost nonfissile and does not have the intermediate value property; however, $f$ does have the weak intermediate value property.
\end{example}

Note that the set of nonfissile points of $f$ is the interval $[0,1)$ which is a dense $G_{\delta}$ subset of $[0,1]$, but $f$ fails to be almost nonfissile because the set of nonfissile points of $G(f)$ is $[0,1) \times \{0\}$ which is not dense in $G(f)$.

\begin{example}  \label{IVPNotANF}
Let $\{C_{r}: r\in\mathbb{Q}\cap[0,1]\}$ denote a collection of Cantor sets in $(0,1]$ such that, for $r > s$, $C_r$ is a subset of $C_s$ that contains no endpoint of $C_s$, and let $f:[0,1]\rightarrow C([0,1])$ be defined by

\vspace*{.5\baselineskip}

$f(t)= \left\{ \begin{array}{ll}
       \{0\} & \text{if } t \notin C_{0} \\
       {[}0,\sup\{r: t \in C_{r}\}{]} & \text{if } t \in C_{0}
       \end{array}\right.$
\vspace{.5\baselineskip}  \\
\noindent Then $f$ is not almost nonfissile, but $f$ does have the intermediate value property.
\end{example}

The function $f$ is the simplest member of a family of functions which, together with their inverse limits, constitute the focus of \cite{DunnRyden2}.  There it is shown that each function in the family has the intermediate value property and periodic cycles of all periods, but generates an hereditarily decomposable tree-like continuum as its inverse limit.  Also in \cite{DunnRyden2} is the construction of a collection of Cantor sets with the properties required for the definition of $f$.



\begin{definition}
\rm{A function $f:X\rightarrow 2^{Y}$ is \emph{light} if for every $y\in[0,1]$, the set $\{x\in[0,1]:y\in f(x)\}$ has no interior.}
\end{definition}

\noindent {\bf Notation.} If $x_{1},x_{2}\in [a,b]$, let $\overline{x_{1}x_{2}}$ denote the closed interval with endpoints $x_{1}$ and $x_{2}$.

\begin{definition}
\rm{If $f:[a,b]\rightarrow2^{[a,b]}$ is upper semicontinuous, we say $f$ is \emph{organic} if for every $x,y\in\invlim\{[a,b],f\}$ such that $\invlim\{[a,b],f\}$ is irreducible between $x$ and $y$, then there exists $n\in\mathbb{N}$ such that $f^{n}(\overline{x_{n}y_{n}})=[a,b]$.}
\end{definition}

We briefly justify our focus on functions with the intermediate value property.

\begin{theorem}\cite{Dunn} \label{InvLim}
Let $f:[0,1]\rightarrow 2^{[0,1]}$ be a surjective, upper semicontinuous function with the weak intermediate value property and a connected graph $G(f)$. Then $\invlim\{[0,1],f\}$ is connected.
\end{theorem}

\noindent {\bf Notation.} Let $a<b$. Define $V_{[a,b]}=[a,b]\times[0,1]$.

\begin{theorem}\label{Intersection}\cite{Dunn}
Let $f:[0,1]\rightarrow 2^{[0,1]}$ be upper semicontinuous. Then the following are equivalent.
\begin{enumerate}
 \item $f$ has the intermediate value property.
 \item For every $a<b$, $G(f)\cap V_{[a,b]}$ is connected and $G(f)\cap V_{[a,b]}=\overline{G(f)\cap V_{(a,b)}}$.
 \end{enumerate}
\end{theorem}

\begin{theorem} \label{IVPCriterion}
Suppose $f:[0,1]\rightarrow 2^{[0,1]}$ is upper semicontinuous. Then $f$ has the intermediate value property if and only if $f$ is weakly continuous and $f(x)$ is connected for each $x$.
\end{theorem}

\begin{proof}
If $f$ has the intermediate value property, then $f$ is weakly continuous and $f(x)$ is connected for each $x$ by Theorem \ref{Intersection}.  To see the converse, suppose $f$ does not have the intermediate value property but that $f$ is weakly continuous.  We show that $f(x)$ fails to be connected for some $x \in [0,1]$.  Since $f$ does not have the intermediate value property, there are $(x_1, y_1), (x_2, y_2) \in G(f)$ and $y$ strictly between $y_1$ and $y_2$ such that $y \not \in f(x)$ for all $x$ strictly between $x_1$ and $x_2$.  There are four cases, all similar, corresponding to the orders of $x_1$ and $x_2$ and of $y_1$ and $y_2$.  We consider only the case in which $x_1 < x_2$ and $y_1 < y_2$.

Since $f$ is weakly continuous from the right at $x_1$, there is $(x'_1, y'_1) \in G(f)$ such that $x_1 < x'_1 < x_2$ and $y'_1 < y$.  Since $f$ is weakly continuous from the left at $x_2$, there is $(x'_2, y'_2) \in G(f)$ such that $x'_1 < x'_2 < x_2$ and $y < y'_2$.  It follows that, for each $x \in [x'_1, x'_2]$, $y \not \in f(x)$.  Since $y'_1 < y < y'_2$ it follows that $V_{[x'_1, x'_2]} \cap G(f)$ is the union of two disjoint compact sets $K_1$ and $K_2$.  Then $\pi_1[K_1] \cup \pi_1[K_2] = [x'_1, x'_2]$.  Consequently, there is $c \in \pi_1[K_1] \cap \pi_1[K_2]$.  It follows that $\{c\}\times f(c)$ is a subset of $K_1 \cup K_2$ that intersects both $K_1$ and $K_2$.  Hence $f(c)$ is not connected.
\end{proof}

\section{Full-Projection Property}  \label{c}

In this section we consider the full-projection property for inverse limits of upper semicontinuous functions with the intermediate value property.  It is shown elsewhere \cite{Ryden} that an inverse limit with upper semicontinuous bonding functions has the full-projection property if and only if its nonfissile points constitute a dense $G_{\delta}$ subset of the inverse limit.  In light of this, it is reasonable to wonder whether an equivalent or even sufficient condition might be to require that the bonding functions of the inverse limit be almost nonfissile.  Alone, almost nonfissile does not suffice; in tandem with surjectivity, lightness, and the intermediate value property, it does.  Theorem~\ref{FPP} establishes this, and Theorem~\ref{FPP2} provides a generalization, that any subcontinuum with nondegenerate projections in all coordinates may also be written as an inverse limit with the full-projection property by restricting the bonding functions appropriately.  These are the main results of Section~\ref{c.2}

In Section~\ref{c.1}, we present results intended to provide intuition regarding the structure of almost nonfissile functions and their graphs.  In Proposition~\ref{c.c}, it is shown that an upper semicontinuous interval function is almost nonfissile if and only if it is irreducible with respect to domain.  B.R.~Williams \cite{Williams} defined ``irreducible with respect to domain" to study the full-projection property.  Iztok Bani\v{c}, Matev\v{z} \v{C}repnjak, Matej Merhar, and Uro\v{s} Milutinovi\'{c} \cite{BanicCrepnjakMerharMilutinovic} studied the property further and introduced variations to Williams's definition.



\subsection{The equivalence of almost nonfissile to irreducibility with respect to domain}  \label{c.1}

\begin{lemma}\label{ANFDomain}
Let $f:[0,1]\rightarrow 2^{[0,1]}$ be an upper semicontinuous function.  Then the set of nonfissile points of $f$ is a $G_{\delta}$ subset of $[0,1]$.   If $\text{int }G(f)=\emptyset$, then it is a dense $G_{\delta}$ subset of $[0,1]$.
\end{lemma}

\begin{proof}
Define $A=\{x\in[0,1]:|f(x)|>1\}$, and, for each $n\in\mathbb{N}$, define $$D_{n}=\left\{x\in[0,1]:{\rm diam}\ f(x)\geq\frac{1}{n}\right\}.$$
As $f$ is upper semicontinuous, $D_{n}$ is closed for each $n$. Note that $A=\bigcup_{n\in\mathbb{N}}D_{n}$, making $A$ an $F_{\sigma}$ set.  It follows that the set of nonfissile points of $f$ is a $G_{\delta}$ set.

We prove the second statement by contraposition.  To that end, suppose the set of nonfissile points of $f$ is not dense or, equivlaently, that $A$ is nonmeager.  Then there is some fixed $n$ such that $D_{n}$ is not nowhere dense, i.e. ${\rm int}\ D_{n}\neq\emptyset$. So there is some nondegenerate interval $[a,b]\subseteq D_{n}$.

Let $\epsilon=\inf_{x\in[a,b]}{\rm diam}\ f(x)\geq\frac{1}{n}$. Then for any $\eta>0$, there exists $z\in[a,b]$ such that $\epsilon\leq {\rm diam}\ f(z)<\epsilon+\eta$. In particular, for $\eta=\frac{\epsilon}{8}$, there is a $z\in[a,b]$ such that ${\rm diam}\ f(z)<\frac{9\epsilon}{8}$. We assume the case $z\in[a,b)$, as the argument for $z=b$ follows a similar argument. Let $c=\min f(z)$ and $d=\max f(z)$. Since $f$ is upper semicontinuous, there is some $\delta>0$ such that if $x\in(z,z+\delta)$, then $f(x)\subseteq \left(c-\frac{\epsilon}{8},d+\frac{\epsilon}{8}\right)$.


Let $x\in(z,z+\delta)$. As ${\rm diam}\ f(x)\geq\epsilon$, $f(x)\subseteq\left(c-\frac{\epsilon}{8},d+\frac{\epsilon}{8}\right)$, and ${\rm diam} \left(c-\frac{\epsilon}{8},d+\frac{\epsilon}{8}\right)<\frac{11\epsilon}{8}$, $f(x)\supseteq\left[c+\frac{\epsilon}{4},d-\frac{\epsilon}{4}\right]$, an interval with nonemtpy interior. As $x$ was arbitrary, the set $$U=\left\{(x,y):z<x<z+\delta \text{ and } y\in\left(c+\frac{\epsilon}{4},d-\frac{\epsilon}{4}\right)\right\}$$ is an open subset of $G(f)$, so $G(f)$ has nonempty interior. Therefore, by contraposition, if ${\rm int} \ G(f)=\emptyset$, then $A$ is meager. So the set of points in $[0,1]$ on which $f$ is single-valued is a dense $G_{\delta}$.
\end{proof}

\begin{definition}
\rm{A function $f:[0,1]\rightarrow 2^{[0,1]}$ is \emph{irreducible with respect to domain} if no closed subgraph of $G(f)$ has full domain, that is, $\pi_{1}[H]\neq [0,1]$ for every closed set $H\subsetneq G(f)$.}
\end{definition}

\begin{proposition}  \label{c.c}
Let $f:[0,1]\rightarrow 2^{[0,1]}$ be an upper semicontinuous function. Then $f$ is almost nonfissile if and only if $f$ is irreducible with respect to domain.
\end{proposition}

\begin{proof}
First note that if $\text{int }G(f) \not = \emptyset$, then $f$ is neither almost nonfissile nor irreducible with respect to domain.  Suppose $\text{int }G(f)=\emptyset$.   Let ${\rm Fi}(f)$ be the set of fissile points of $G(f)$ and $A=G(f)\setminus {\rm Fi}(f)$, i.e. the set of nonfissile points of $G(f)$. By Lemma \ref{ANFDomain}, $\pi_{1}[A]$ is a dense $G_{\delta}$ subset of $[0,1]$. Then $\overline{A}$ is a closed subgraph of $G(f)$ with full domain. So if $f$ is irreducible with respect to domain, $\overline{A}=G(f)$, making $f$ almost nonfissile. Conversely, if $f$ is almost nonfissile, then as $A$ is composed of nonfissile points, any closed subgraph with full domain must contain $A$ and hence contains $\overline{A}$. Thus if $f$ is almost nonfissile and $H$ is a closed subgraph of $G(f)$ with full domain, $H\supseteq \overline{A}=G(f)$, making $f$ irreducible with respect to domain.
\end{proof}

\subsection{The full-projection property in inverse limits of maps with the intermediate value property}  \label{c.2}

\begin{theorem}  \label{c.d}  \cite{Ryden}
Suppose $\{X_n, f_n\}$ is an inverse sequence and $X = \ilim \{X_n, f_n\}$.  Then $X$ has the full-projection property if and only if the set of fissile points of $X$ is a meager $F_{\sigma}$ set.
\end{theorem}

\begin{lemma}\label{Intervals}
Suppose $f:[0,1]\rightarrow 2^{[0,1]}$ is a surjective, almost nonfissile, upper semicontinuous map with the intermediate value property. If $f(x)$ is nondegenerate for some interior point $x$ of $[0,1]$, then there are sequences $L_{1}$, $L_{2}$, \dots, and $R_{1}$, $R_{2}$, \dots of nondegenerate closed subintervals of $[0,1]$ such that
\begin{enumerate}
\item $z<x$ for all $z\in\bigcup L_{n}$ and $z>x$ for all $z\in\bigcup R_{n}$,
\item $\lim L_{n}=\{x\}$ and $\lim R_{n} =\{x\}$,
\item $\lim f[L_{n}]=f(x)$ and $\lim f[R_{n}]=f(x)$.
\end{enumerate}
\end{lemma}

\begin{proof}
We construct the sequence $L_{1}$, $L_{2}$,\dots only and note that the construction of $R_{1}$, $R_{2}$,\dots is similar. Let $a$ and $b$ denote the points such that $f(x)=[a,b]$. Since upper semicontinuous maps with the intermediate value property are weakly continuous by Theorem \ref{IVPCriterion}, there are sequences $\alpha_{1}$, $\alpha_{2}$,\dots; $\beta_{1}$, $\beta_{2}$,\dots; $a_{1}$, $a_{2}$,\dots; and $b_{1}$, $b_{2}$, \dots such that each of the following is true:
\begin{itemize}
\item $a_{n}\in f(\alpha_{n})$ for all $n$,
\item $b_{n}\in f(\beta_{n})$ for all $n$,
\item $\{\alpha_{n}\}$ and $\{\beta_{n}\}$ converge to $x$ from the left,
\item $\{a_{n}\}$ and $\{b_{n}\}$ converge to $a$ and $b$ respectively.
\end{itemize}
Furthermore, since $f$ is almost nonfissile, the sequences may be chosen so that $f(\alpha_{n})=\{a_{n}\}$ and $f(\beta_{n})=\{b_{n}\}$. It follows that, for sufficiently large $n$, $a_{n}$ and $b_{n}$ are distinct. Finally, taking subsequences if necessary, the sequences may be chosen so that $a_{n}<\frac{a+b}{2}<b_{n}$ for each $n\in\mathbb{N}$ and $\alpha_{n},\beta_{n}<\alpha_{n+1},\beta_{n+1}$ for each $n\in\mathbb{N}$.

Since $a_{n}\neq b_{n}$ for each $n\in\mathbb{N}$, it follows that $\alpha_{n}\neq\beta_{n}$ for each $n\in\mathbb{N}$. For each $n\in\mathbb{N}$, define $L_{n}$ to be the nondegenerate closed interval with endpoints $\alpha_{n}$ and $\beta_{n}$. Then $L_{1}$, $L_{2}$,\dots satisfies $(1)$ and $(2)$. To see that it satisfies $(3)$, note that $\lim\ \inf f[L_{n}]$ contains both $a$ and $b$, and hence $f(x)$, since $a_{n}\rightarrow a$ and $b_{n}\rightarrow b$ as $n\rightarrow\infty$. On the other hand, $\lim\ \sup f[L_{n}]\subseteq f(x)$ since the graph of $f$ is closed. Hence $\lim f[L_{n}]=f(x)$, and $\{L_{n}\}$ satisfies $(3)$.
\end{proof}

\begin{lemma}\label{OpenSet}
Suppose $f:[0,1]\rightarrow 2^{[0,1]}$ is an almost nonfissile upper semicontinuous function with the intermediate value property. If $y\in f(x)$, and $D_{x}$ and $D_{y}$ are open sets such that $y\in D_{y}$ and $x\in \overline{D_{x}}$, then there is an open subset $D$ of $D_{x}$ such that $f[D]\subset D_{y}$.
\end{lemma}

\begin{proof}
Since $f$ is weakly continuous from both the left and the right by Theorem \ref{IVPCriterion}, there is a point $x_{1}\in D_{x}$ such that $f(x_{1})$ intersects $D_{y}$. Since $f$ is almost nonfissile, there is a nonfissile point $x_{2} \in D_{x}$, i.e. that $f(x_{2})=\{y_{2}\}\subseteq D_{y}$. Put $D=\{x\in[0,1]:f(x)\subset D_{y}\}\cap D_{x}$. Then $D$ is a nonempty open subset of $D_{x}$ that contains $x_{2}$. Furthermore, $f[D]\subset D_{y}$.
\end{proof}

\begin{lemma}\label{Delta}
Suppose $f:[0,1]\rightarrow 2^{[0,1]}$ is a light, almost nonfissile, upper semicontinuous map with the intermediate value property. If $G$ is a $G_{\delta}$ subset of $D$ for some open subset $D$ of $[0,1]$ then $\{x\in [0,1]:f(x)\subseteq G\}$ is a $G_{\delta}$ subset of $[0,1]$. Furthermore, if $G$ is dense in $D$, then $\{x\in[0,1]: f(x)\subseteq G\}$ is dense in $\{x\in [0,1]:f(x)\subseteq D\}$.
\end{lemma}

\begin{proof}
There are open sets $G_{1}$, $G_{2}$, \dots such that $\bigcap G_{n}=G$. Since $f$ is upper semicontinuous, $\{x\in [0,1]:f(x)\subset G_{n}\}$ is open in $[0,1]$. Note that $\bigcap\{x\in [0,1]:f(x)\subset G_{n}\} = \{x\in[0,1]:f(x)\subset\bigcap G_{n}\}=\{x\in[0,1]:f(x)\subset G\}$. It follows that $\{x\in[0,1]:f(x)\subset G\}$ is a $G_{\delta}$ set.

Suppose further that $G$ is dense in some open set $D$. Replacing $G_{n}$ with $G_{n}\cap D$ for each $n\in\mathbb{N}$ if necessary, the open sets $G_{n}$ may be taken to be open subsets of $D$ for which $\bigcap G_{n}=G$. Note that $G_{n}$ is dense in $D$ for each $n\in\mathbb{N}$. Suppose $U$ is an open interval in $\{x\in[0,1]: f(x)\subset D\}$. Since $f$ is light and has the intermediate value property, $f[U]$ is a nondegenerate interval in $D$. Then $\text{int } f[U]$ contains a point of $G_{n}$. It follows that there is a point $u$ of $U$ and a point $w$ of $\text{int } f[U]\cap G_{n}$ such that $w\in f(u)$. Since $f$ is almost nonfissile, $u$ and $w$ may be chosen so that $f(u)=\{w\}$. It follows that $G_{n}$ contains $f(u)$ and $U$ contains a point of $\{x\in[0,1]:f(x)\subset G_{n}\}$. Hence $\{x\in[0,1]: f(x)\subset G_{n}\}$ is a dense open subset of $\{x\in[0,1]:f(x)\subset D\}$. As this is true for each $n\in\mathbb{N}$, $\{x\in[0,1]:f(x)\subset G\}$ is a dense $G_{\delta}$ subset of $\{x\in[0,1]:f(x)\subset D\}$.
\end{proof}

\begin{lemma} \label{c.i}
Suppose $f:[0,1]\rightarrow 2^{[0,1]}$ is a surjective, light, almost nonfissile, upper semicontinuous map with the intermediate value property. If $y\in f(x)$ and $D_{y}$ is an open set such that $y\in\overline{D_{y}\cap(-\infty,y)}\cap \overline{D_{y}\cap(y,\infty)}$, then there is an open set $D_{x}$ such that $x\in \overline{D_{x}\cap(-\infty,x)}\cap\overline{D_{x}\cap(x,\infty)}$ and such that $f[D_{x}]\subset D_{y}$.
\end{lemma}

\begin{proof}
First suppose $f(x)$ is nondegenerate. Then, by the Lemma \ref{Intervals}, there are sequences $L_{1}$, $L_{2}$,\dots and $R_{1}$, $R_{2}$,\dots of nondegenerate closed subintervals of $[0,1]$ such that
\begin{enumerate}
\item $z<x$ for all $z\in \bigcup L_{n}$, and $z>x$ for all $z\in\bigcup R_{n}$,
\item $\lim L_{n}=\{x\}$, and $\lim R_{n} = \{x\}$, and
\item $\lim f[L_{n}]=f(x)$ and $\lim f[R_{n}]=f(x)$.
\end{enumerate}
Since $f(x)$ is a nondegenerate interval containing $y$, at least one of $f(x)\cap(-\infty,y)$ and $f(x)\cap(y,\infty)$ is a nondegenerate interval with one endpoint equal to $y$, say $f(x)\cap (y,\infty)$. Since $y\in\overline{D_{y}\cap(y,\infty)}$, every open interval whose left endpoint is $y$ contains a point of $D_{y}\cap(y,\infty)$. It follows that $\text{int} f(x)\cap (y,\infty)$ contains a point of $D_{y}\cap(y,\infty)$. Hence $\text{int} f(x)\cap D_{y}\cap (y,\infty)$ contains an open interval $(y_{1},y_{2})$; furthermore, $y_{1}$ and $y_{2}$ may be chosen so that neither of them is an endpoint of $f(x)$. Since $f[L_{n}]$ and $f[R_{n}]$ are connected for each $n\in\mathbb{N}$ by dint of the intermediate value property and since $\lim f[L_{n}]=\lim f[R_{n}]=f(x)$, it follows that there is $N\in\mathbb{N}$ such that $f[L_{n}]$ and $f[R_{n}]$ both contain $(y_{1},y_{2})$ for each $n\geq N$. Hence, for each $n\geq N$, some point of $L_{n}$ has an image that intersects $(y_{1},y_{2})$. As $f$ is weakly continuous from both the left and the right by Theorem \ref{IVPCriterion}, there are, for each $n\geq N$, points $l_{n}\in\text{int } L_{n}$ and $\tilde{l}_{n}\in(y_{1},y_{2})$ such that $\tilde{l}_{n}\in f(l_{n})$. Furthermore, since $f$ is almost nonfissile, $l_{n}$ and $\tilde{l}_{n}$ may be chosen so that $l_{n}$ is a nonfissile point of $f$. Then $(y_{1},y_{2})$ is an open set containing $f(l_{n})$. Hence $\{x\in[0,1]:f(x)\subset(y_{1},y_{2})\}$ is an open set containing $l_{n}$. For each $n\geq N$, put $U_{n}=\text{int } L_{n}\cap \{x\in[0,1]:f(x)\subset (y_{1},y_{2})\}$. Then $U_{n}\subset L_{n}$, and $f[U_{n}]\subset(y_{1},y_{2})\subset D_{y}$. Similarly, for $n\geq N$, there are open sets $V_{n}\subset R_{n}$ such that $f[V_{n}]\subset D_{y}$. Finally, put $D_{x}=(\bigcup_{n\geq N}U_{n})\cup(\bigcup_{n\geq N}V_{n})$. Note that $f[D_{x}]\subset D_{y}$. Thus it remains only to show that $x\in \overline{D_{x}\cap(-\infty,x)}\cap\overline{D_{x}\cap(x,\infty)}$.

To that end note that, by $(1)$ and the fact that $U_{n}\subset L_{n}$ and $V_{n}\subset R_{n}$ for each $n\geq N$, we have $D_{x}\cap(-\infty,x)=\bigcup_{n\geq N}U_{n}$ and $D_{x}\cap(x,\infty)=\bigcup_{n\geq N}V_{n}$. It follows from $(2)$ that $x\in\overline{\bigcup_{n\geq N}U_{n}}$ and $x\in\overline{\bigcup_{n\geq N}V_{n}}$. Consequently, $x\in \overline{D_{x}\cap(-\infty,x)}\cap\overline{D_{x}\cap(x,\infty)}$.

Now suppose $f(x)$ is degenerate, that is, suppose $f(x)=\{y\}$. Suppose $n\in\mathbb{N}$ is given, and consider the interval $\left(x,x+\frac{1}{n}\right)$. Since $f$ is light and upper semicontinuous, $f\left(x,x+\frac{1}{n}\right)$ is a nondegenerate interval. Since the graph of $f$ is closed, $y\in\overline{f\left(x,x+\frac{1}{n}\right)}$. Since $f\left(x,x+\frac{1}{n}\right)$ is an interval, this is equivalent to $y\in\overline{\text{int } f\left(x,x+\frac{1}{n}\right)}$. It follows that $\text{int } f\left(x,x+\frac{1}{n}\right)\cap D_{y}$ is nonempty. By Lemma \ref{OpenSet}, there is an open subset $V_{n}$ of $\left(x,x+\frac{1}{n}\right)$ such that $f[V_{n}]\subset D_{y}$. Similarly, there is an open subset $U_{n}$ of $\left(x-\frac{1}{n},x\right)$ such that $f[U_{n}]\subset D_{y}$. Thus $U_{n}$ and $V_{n}$ are defined for $n\in\mathbb{N}$. Put $D_{x}=(\bigcup_{n\in\mathbb{N}}U_{n})\cup (\bigcup_{n\in\mathbb{N}}V_{n})$. Then $f[D_{x}]\subset D_{y}$ and $x\in\overline{(\bigcup_{n\in\mathbb{N}}U_{n})}\cap\overline{(\bigcup_{n\in\mathbb{N}}V_{n})}=\overline{D_{x}\cap (-\infty,x)}\cap\overline{D_{x}\cap(x,\infty)}$.
\end{proof}

\begin{lemma}\label{Comeager}
Suppose $\{[0,1],f_{n}\}$ is an inverse sequence where, for each $n\in\mathbb{N}$, $f_{n}$ is a surjective almost nonfissile, light, upper semicontinuous map with the intermediate value property. For each $N \in {\mathbb N}$, if $x \in \invlim\{[0,1],f_{n}\}$ and $U_{0}$, $U_{1}$, \dots, $U_{N}$ are open sets containing $x_{0}$, $x_{1}$, \dots, $x_{N}$ respectively, then there are open subsets $D_{0}$, $D_{1}$, \dots, $D_{N}$ of $U_{0}$, $U_{1}$, \dots, $U_{N}$ respectively such that
\begin{enumerate}
\item $x_{n}\in\overline{D_{n}\cap(-\infty,x_{n})}\cap\overline{D_{n}\cap(x_{n},\infty)}$ for $n=0,1,\dots, N$,
\item $f_{n}[D_{n}]\subset D_{n-1}$ for $n=1,2,\dots, N$, and
\item $z_{N}$, $f_{N-1}^{N}(z_{N})$, \dots, $f_{1}^{N}(z_{N})$ are nonfissile points of $f_{N}$, $f_{N-1}$, \dots, $f_{1}$ respectively for all $z_{N}$ in some comeager subset of $D_{N}$.
\end{enumerate}
\end{lemma}

\begin{proof}
The proof is by induction. First consider $N=1$. Suppose $x\in\invlim\{[0,1],f_{n}\}$, and suppose $U_{0}$ and $U_{1}$ are open sets containing $x_{0}$ and $x_{1}$ respectively. Put $D_{0}=U_{0}$. Note that $D_{0}$ satisfies the requirement in (1). By the Lemma~\ref{c.i}, there is an open set $\tilde{D}_{1}$ such that $x_{1}\in\overline{\tilde{D}_{1}\cap(-\infty,x_{1})}\cap\overline{\tilde{D}_{1}\cap(x_{1},\infty)}$ and $f[\tilde{D}_{1}]\subset D_{0}$. Put $D_{1}=U_{1}\cap \tilde{D}_{1}$. Then $D_{0}$ and $D_{1}$ satisfy (1) and (2). The set of nonfissile points of $f_{1}$ is a $G_{\delta}$ subset of $[0,1]$ by Lemma~\ref{ANFDomain} and dense in $[0,1]$ since $f_{1}$ is almost nonfissile.  Since $D_{1}$ is open, the set of nonfissile points of $f_{1}$ that lie in $D_{1}$ is a comeager subset of $D_{1}$. Hence (3) holds, and the result is true for $N=1$.

Suppose that the result is true for $N=k$ for some $k\geq 1$, and consider $n=k+1$. Suppose $x\in\invlim\{[0,1],f_{n}\}$, and suppose $U_{0}$, $U_{1}$, \dots, $U_{k+1}$ are open sets containing $x_{1}$, $x_{2}$, \dots, $x_{k+1}$ respectively. Since the result holds for $N=k$, there are open subsets $D_{0}$, $D_{1}$, \dots, $D_{k}$ of $U_{0}$, $U_{1}$, \dots, $U_{k}$ that satisfy (1), (2), and (3). By Lemma~\ref{c.i}, there is an open set $D_{k+1}$ such that $x_{k+1}\in\overline{D_{k+1}\cap(-\infty,x_{k+1})}\cap\overline{D_{k+1}\cap(x_{k+1},\infty)}$ and $f_{k+1}[D_{k+1}]\subset D_{k}$. Replacing $D_{k+1}$ with $D_{k+1}\cap U_{k+1}$ if necessary, we may assume $D_{k+1}\subset U_{k+1}$. Note that $D_{k+1}$ satisfies (1) and (2). Thus it remains to show that $D_{k+1}$ satisfies (3).

For each $n \in {\mathbb N}$, denote the set of fissile points of $f_n$ by ${\rm Fi}(f_n)$.  By Lemma \ref{ANFDomain}, ${\rm Fi}(f_{n})$ is an $F_{\sigma}$ set for each $n=1,2,\dots,k+1$. Since $f_{n}^{k+1}$ is upper semicontinuous for each $n$, $(f_{n}^{k+1})^{-1}({\rm Fi}(f_{n}))$ is an $F_{\sigma}$ set for each $n=1,2,\dots, k+1$. Hence $\bigcup_{n=1}^{k+1}(f_{n}^{k+1})^{-1}({\rm Fi}(f_{n}))$ is an $F_{\sigma}$ set. Equivalently, $\{z\in D_{k+1}: z, f_{k}^{k+1}(z),\dots, f_{1}^{k+1}(z) \text{ are nonfissile}$ $\text{points of $f_{k}$, $f_{k-1}$, \dots, $f_{1}$ respectively}\}$ is a $G_{\delta}$ set. Denote it by $A$, and note that $A\cap D_{k+1}$ is a $G_{\delta}$ subset of $D_{k+1}$. To see that $A\cap D_{k+1}$ is dense in $D_{k+1}$, suppose $D$ is an open interval in $D_{k+1}$. Since $f_{k+1}$ is light and has the intermediate value property, $f_{k+1}[D]$ is a nondegenerate interval in $D_{k}$. Since $D_{k}$ satisfies (3),
$\{z\in D_{k}: z,f_{k-1}^{k}(z),\dots,f_{1}^{k}(z) \text{ are non-}$ $\text{fissile points of $f_{k}$, $f_{k-1}$, \dots, $f_{1}$ respectively}\}$ is a dense $G_{\delta}$ set in $D_{k}$. Denote this set by $G$. Then, by Lemma \ref{Delta}, $\{x\in[0,1]: f_{k+1}(x)\subset G\}$ is a dense $G_{\delta}$ set in $\{x\in[0,1]:f_{k+1}(x)\subset D_{k}\}$. Since $D_{k+1}\subset \{x\in[0,1]:f_{k+1}(x)\subset D_{k}\}$, it follows that $D_{k+1}\cap\{x\in[0,1]:f_{k+1}(x)\subset G\}$ is a dense $G_{\delta}$ subset of $D_{k+1}$.  The set of nonfissile points of $f_{k+1}$ in $D_{k+1}$ is also a dense $G_{\delta}$ subset of $D_{k+1}$ by Lemma~\ref{ANFDomain} and the fact that $f_{k+1}$ is almost nonfissile.  Put $A=([0,1]-{\rm Fi}(f_{k+1}))\cap D_{k+1}\cap \{x\in[0,1]:f_{k+1}(x)\subset G\}$. Then $A$ is a dense $G_{\delta}$ subset of $D_{k+1}$, and, for each $z\in A$, $z$, $f_{k}^{k+1}(z)$, $f_{k-1}^{k+1}(z)$, \dots, $f_{1}^{k+1}(z)$ are nonfissile points of $f_{k+1}$, $f_{k}$, \dots, $f_{1}$ respectively. Hence $D_{k+1}$ satisfies (3), and the inductive step is complete.
\end{proof}


\begin{theorem} \label{FPP}
Suppose $\{[0,1],f_{n}\}$ is an inverse sequence where, for each $n \in \mathbb{N}$, $f_{n}:[0,1] \rightarrow 2^{[0,1]}$ is a surjective, light, almost nonfissile, upper semicontinuous map with the intermediate value property. Then $\invlim\{[0,1],f_{n}\}$ has the full-projection property.
\end{theorem}

\begin{proof}
Denote $\invlim\{[0,1],f_{n}\}$ by $X$. By Theorem~\ref{c.d}, it suffices to show that the set of nonfissile points of $X$ is dense in $X$.  For each $n \in {\mathbb N}$, denote $\{x \in X: |f_n(x_n)| = 1\}$ by $\sim {\rm Fi}_n(X)$, and note that the set of nonfissile points of $X$ is $\sim {\rm Fi}_{1}(X) \cap \sim {\rm Fi}_{2}(X) \cap \sim {\rm Fi}_{3}(X) \cap ...$.  Since $\sim {\rm Fi}_{n}(X)$ is a $G_{\delta}$ subset of $X$ for each $n$, it suffices to show that $\sim {\rm Fi}_{1}(X)\cap \sim {\rm Fi}_{2}(X)\cap \dots\cap \sim {\rm Fi}_{n}(X)$ is dense in $X$ for each $n\geq 1$.

To that end, suppose $n$ is given and $D$ is a nonempty basic open set in $X$. Then $D$ has the form $D=D_{1} \times D_{2} \times \dots \times D_{m} \times [0,1] \times [0,1] \times \dots$ where $D_{i}$ is an open subset of $[0,1]$ for $i=1,2,\dots, m$, and where $m\geq n$. We must show that $D$ contains a point of $\sim {\rm Fi}_{1}(X) \cap \sim {\rm Fi}_{2}(X) \cap \dots \cap \sim {\rm Fi}_{n}(X)$, to which end it suffices to show that $D$ contains a point of $\sim {\rm Fi}_{1}(X) \cap \sim {\rm Fi}_{2}(X) \cap \dots \cap \sim {\rm Fi}_{m}(X)$. This is a consequence of Lemma \ref{Comeager}.
\end{proof}

\begin{theorem} \label{FPP2}
Suppose $\{[0,1],f_{n}\}$ is an inverse sequence where, for each $n\in\omega$, $f_{n}:[0,1] \rightarrow 2^{[0,1]}$ is a surjective, light, almost nonfissile, upper semicontinuous map with the intermediate value property. Let $K$ be a subcontinuum of $\invlim\{[0,1],f\}$ such that $\pi_{n}[K]$ is nondegenerate for each $n$. Then $K$ can be written as the inverse limit of its projections and has the full-projection property.
\end{theorem}

\begin{proof}
For each $n\in\omega$, let $K_{n}=\pi_{n}[K]$. Then $f_{n}$ maps $\pi_{n+1}[K]$ onto $\pi_{n}[K]$. Denote by $f'_{n}$ the restriction of $f_{n}$, $f_{n}|_{\pi_{n+1}[K]}^{\pi_{n}[K]}:\pi_{n+1}[K]\rightarrow 2^{\pi_{n}[K]}$. Note $f'_{n}$ inherits the properties of $f_{n}$ given in the hypothesis.

Define $K' = \invlim\{\pi_{n}[K], f'_{n}\}$.  By Theorems~\ref{InvLim} and \ref{FPP}, $K'$ is a subcontinuum of $\invlim\{[0,1], f_{n}\}$ with the full-projection property.

To show $K'=K$, let $x\in K$. Then for all $n$, $\pi_{n}(x)\in \pi_{n}(K)$ and $\pi_{n}(x)\in f(\pi_{n+1}(x))$ for all $n$. So $\pi_{n}(x)\in f'_{n}(\pi_{n+1}(x))$, i.e. $x\in K'$. Therefore $K\subseteq K'$. But $\pi_{n}(K)=\pi_{n}(K')$ for all $n$. Then as $K'$ has the full-projection property, $K'=K$.
\end{proof}

\section{Relationship Between Periodicity and Indecomposability}  \label{d}

We now turn to the connection between periodicity in an upper semicontinuous function $f:[0,1]\rightarrow 2^{[0,1]}$ with the intermediate value property and indecomposability in the corresponding inverse limit.  In particular, we generalize a connection established in the classical setting by Barge and Martin \cite[Theorems 1 \& 7]{BargeMartin1}.

In Section~\ref{d.1}, we examine how a periodic cycle of $f$ with period not a power of two gives rise to an indecomposable subcontinuum of the inverse limit. The primary result is Theorem \ref{Interior}.  The proof leans heavily on the intermediate value property, appealing to both the Sarkovskii order and the full-projection property, each of which holds in a context involving the intermediate value property (Theorems~\ref{FPP} and \ref{GenSarkovskii}).

We then explore a pseudo converse in Section~\ref{d.2}, that is, how the indecomposability of $\invlim\{[0,1],f\}$ gives rise to a periodic cycle of $f$ with period not a power of two. This subsection focuses on organic maps and has Theorem \ref{IndecomposablePeriodic} as its main result.



\subsection{Periodicity giving rise to indecomposability} \label{d.1}

A.N.~Sarkovskii \cite{Sarkovskii} introduced the following ordering of the positive integers, now known as the Sarkovskii ordering, and used it to show that, for any continuous mapping of the real line into itself, the existence of a cycle of period $m$ follows from the existence of a cycle of period $n$ if and only if $n \preceq m$.
\[
3\prec5\prec7\prec9\prec\dots\]\[
3\cdot 2\prec 5\cdot 2\prec 7\cdot 2\prec 9\cdot 2\prec\dots\]\[
3\cdot 2^{2}\prec 5\cdot 2^{2}\prec 7\cdot 2^{2}\prec 9\cdot 2^{2}\prec \dots\]
\[\dots\]\[
\dots2^{4}\prec 2^{3}\prec 2^{2}\prec 2\prec 1
\]
The following theorem, extends one direction of Sarkovskii's Theorem to upper semicontinuous set valued functions with the intermediate value property.

\begin{theorem}\label{GenSarkovskii} \cite{OteyRyden}
Let $f:[0,1]\rightarrow2^{[0,1]}$ be upper semicontinuous and have the intermediate value property. If $f$ has a cycle of period $n$, then $f$ has cycles of every period $m$ such that $n\prec m$.
\end{theorem}

\begin{lemma}\label{ConstantInterval}
Let $f:[0,1]\rightarrow 2^{[0,1]}$ be upper semicontinuous, surjective, and almost nonfissile and $G(f)$ be connected and have empty interior. If there is some $y\in[0,1]$ and a nondegenerate interval $I\subseteq [0,1]$ such that $y\in f(x)$ for every $x\in I$, then $f$ is constant and single valued on $I$ and $f[I]=\{y\}$.
\end{lemma}

\begin{proof}
Let $x\in I$ and $y'\in f(x)$. Then either $x>\inf I$ or $x<\sup I$. The two cases proceed similarly, so we shall prove the result for $x>\inf I$. As $f$ is weakly continuous, there is a sequence $\{ (x_{n},y_{n})\}_{n\in\omega}$ in $G(f)$ converging to $(x,y')$ such that $x_{n}\in I$ and $x_{n}<x$. Since $f$ is almost nonfissile, we may choose each $(x_{n},y_{n})$ so that $x_{n}$ is a nonfissile point of $f$. Thus $f(x_{n})=\{y_{n}\}$ for all $n$. But $y\in f(x_{n})$, so $y_{n}=y$ for all $n$. As $y_{n}\rightarrow y'$, this implies $y'=y$. As $x$ and $y'$ were arbitrary, $f[I]=\{y\}$.
\end{proof}

\begin{theorem}\label{PeriodicIndecomposable}
Let $f:[0,1]\rightarrow 2^{[0,1]}$ be upper semicontinuous, surjective, almost nonfissile, light, and have the intermediate value property, and $G(f)$ have empty interior. If $f$ has a periodic orbit of period not a power of 2, then $\invlim\{[0,1],f\}$ contains an indecomposable subcontinuum.
\end{theorem}

\noindent {\bf Remark.}
A natural question arising from Theorem~\ref{PeriodicIndecomposable} is whether it remains true without the assumption that the bonding function be almost nonfissile.  It does not.  Consider the function $f$ defined in Example~\ref{IVPNotANF}. We show in \cite{DunnRyden2} that $f$ is part of a family of upper semicontinuous, surjective functions with the intermediate value property and periodic cycles of all periods that are not almost nonfissile and have hereditarily decomposable inverse limits.

Although $f$ is not light, other members of the family are.  They can be obtained by tweaking the value of $f$ on the open intervals in the complement of $C_0$ so that instead of being identically zero, they are light but sufficiently small.  Such examples play an important role in our understanding of the connection between periodicity and indecomposability in generalized inverse limits of $[0,1]$.  In particular, they show that the assumption of almost nonfissile in Theorems~\ref{MainResult}(\ref{MainResult.1}), \ref{PeriodicIndecomposable}, and \ref{Interior} cannot be dropped.


\begin{proof}[Proof of Theorem \ref{PeriodicIndecomposable}]
Suppose $f$ has an orbit of period $n\cdot2^{k}$. By the Theorem \ref{GenSarkovskii}, there is a periodic orbit of $f$ with period $3\cdot 2^{k+1}$. Let $x\in\invlim\{[0,1],f\}$ be the point that models this orbit. Then $(x_{3\cdot2^{k+1}-1},\dots, x_{1},x_{0})$ is a cycle $f$ of period $3\cdot 2^{k+1}$. Let $h$ be the forgetful shift on $\invlim\{[0,1],f\}$. Then $x$ has a period 3 cycle under $h^{2^{k+1}}$, namely $\left(x,h^{2^{k+1}}(x),h^{2^{k+2}}(x)\right)$. To show this, suppose to the contrary that $x$ does not have a period 3 cycle. By the construction of $x$, $h^{3\cdot2^{k+1}}(x)=x$. So either $h^{2^{k+1}}(x)=x$ or $h^{2^{k+2}}(x)=x$. If $h^{2^{k+1}}(x)=x$, then for all $n$, $x_{n+2^{k+1}}=x_{n}$, contradicting the fact $(x_{0},x_{1},\dots,x_{3\cdot2^{k+1}-1})$ is an orbit of period $3\cdot 2^{k+1}$. By a similar argument, $h^{2^{k+2}}(x)\neq x$. Thus $x$ has an orbit of period 3 under $h^{2^{k+1}}$.

Let $S$ be a subcontinuum of $\invlim\{[0,1],f\}$ that is irreducible about $x$, $h^{2^{k+1}}(x)$, and $h^{2^{k+2}}(x)$. By Theorem \ref{FPP2}, there are restrictions $f'_{n}$ of $f$ such that each $f'_{n}$ inherits the properties of $f$ listed in the hypothesis, $S=\invlim\{\pi_{n}(S),f'_{n}\}$, and $S$ has the full-projection property. We show that $S$ is indecomposable by showing it is irreducible about any two points of $x$, $h^{2^{k+1}}(x)$, and $h^{2^{k+2}}(x)$.

By way of contradiction, suppose $S$ is not irreducible between two points of $\{x,h^{2^{k+1}}(x), h^{2^{k+2}}(x)\}$, say $x$ and $h^{2^{k+1}}(x)$. Then there is a proper subcontinuum $H\subsetneq S$ containing $x$ and $h^{2^{k+1}}(x)$. So $h^{2^{k+2}}(x)\notin H$ as $S$ is irreducible about $x$, $h^{2^{k+1}}(x)$, and $h^{2^{k+2}}(x)$.

Since $(x_{3\cdot2^{k+1}-1},\dots, x_{1},x_{0})$ is a cycle $f$ of period $3\cdot 2^{k+1}$, there is some $i\in\{0,1,\dots, 2^{k+1}-1\}$ such that for all $n\in\mathbb{N}$, $\pi_{3n\cdot 2^{k+1}+i}(x)\neq\pi_{3n\cdot 2^{k+1}+i}(h^{2^{k+1}}(x))$. As $h^{2^{k+1}}$ permutes $x$, $h^{2^{k+1}}(x)$, and $h^{2^{k+2}}(x)$, there is some $j\in\{0,1,2\}$ such that for all $n\in\mathbb{N}$, $\pi_{(3n+j)2^{k+1}+i}(h^{2^{k+2}}(x))$ is between $\pi_{(3n+j)2^{k+1}+i}(x)$ and $\pi_{(3n+j)2^{k+1}+i}(h^{2^{k+1}}(x))$. Furthermore, $\pi_{(3n+j)2^{k+1}+i}(h^{2^{k+2}}(x))$ is distinct from at least one of $\pi_{(3n+j)2^{k+1}+i}(x)$ or $\pi_{(3n+j)2^{k+1}+i}(h^{2^{k+1}}(x))$. So $\pi_{(3n+j)2^{k+1}+i}[H]$ is nondegenerate for each $n$, and $\pi_{(3n+j)2^{k+1}+i}(h^{2^{k+2}}(x))\in\pi_{(3n+j)2^{k+1}+i}[H]$. As $f$ is weakly continuous and almost nonfissile, there is a sequence of nonfissile points $\{(x_{k},y_{k})\}_{k\in\mathbb{N}}$ of $G(f)$ such that $x_{k}\in\pi_{(3n+j)2^{k+1}+i}[H]$ and \[(x_{k},y_{k})\rightarrow (\pi_{(3n+j)2^{k+1}+i}(h^{2^{k+2}}(x)), \pi_{(3n+j)2^{k+1}+i-1}(h^{2^{k+2}}(x))).\] Since $f(x_{k})=\{y_{k}\}$, it follows that $y_{k}\in\pi_{(3n+j)2^{k+1}+i-1}[H]$. Then $\pi_{(3n+j)2^{k+1}+i-1}(h^{2^{k+2}}(x))\in \pi_{(3n+j)2^{k+1}+i-1}[H]$ because \\ $y_{k} \rightarrow \pi_{(3n+j)2^{k+1}+i-1}(h^{2^{k+2}}(x))$ and $\pi_{(3n+j)2^{k+1}+i-1}[H]$ is closed. Furthermore, since $f$ is almost nonfissile and light and $\pi_{(3n+j)2^{k+1}+i}[H]$ is nondegenerate, $\pi_{(3n+j)2^{k+1}+i-1}[H]$ is nondegerate by Lemma \ref{ConstantInterval}.

Proceeding inductively, we see that $\pi_{l}[H]$ is nondegenerate and \\ $\pi_{l}(h^{2^{k+2}}(x))\in\pi_{l}[H]$ for all $l\leq (3n+j)2^{k+1}+i$. As this holds for any $n\in\mathbb{N}$, $\pi_{l}(h^{2^{k+2}}(x))\in\pi_{l}[H]$ for all $l\in\mathbb{N}$. Since $H$ is the inverse limit of its own projections by Theorem \ref{FPP2}, $h^{2^{k+2}}(x)\in H$, a contradiction. Therefore $S$ is irreductible about any two points of $\{x,h^{2^{k+1}}(x), h^{2^{k+2}}(x)\}$ and is indecomposable.
\end{proof}

\begin{theorem} \label{Interior}
Suppose $f:[0,1]\rightarrow 2^{[0,1]}$ is upper semicontinuous, surjective, has the intermediate value property, and has an orbit of period not a power of 2. If $f|_{[0,1]\setminus \pi_{1}(int(G(f)))}$ is almost nonfissile and light, then $\invlim\{[0,1],f\}$ contains an indecomposable subcontinuum.
\end{theorem}

\begin{proof}
If $\text{int }G(f)=\emptyset$, then the conclusion follows form Theorem \ref{PeriodicIndecomposable}. Suppose $\text{int }G(f)\neq\emptyset$. Let $(x_{0},x_{1},\dots,x_{p-1})$ be a cycle of $f$ where $p$ is not a power of 2. It is sufficient to show there is a map $g:[0,1]\rightarrow 2^{[0,1]}$ that is upper semicontinuous, almost nonfissile, and light, that has the intermediate value property and retains $(x_{0},x_{1},\dots,x_{p-1})$ as a periodic cycle, and such that $G(g)$ has empty interior and $G(g)\subseteq G(f)$. Then $\invlim\{[0,1],g\}$ is a subcontinuum of $\invlim\{[0,1],f\}$ that contains an indecomposable subcontinuum by Theorem \ref{PeriodicIndecomposable}.

Note $\pi_{1}[\text{int }G(f)]$ is an open subset of $[0,1]$. Let $\{O_{n}\}_{n\in\mathbb{N}}$ be an enumeration of the components of $\pi_{1}[\text{int }G(f)]$. We construct $g$ as follows: if $x\in[0,1]\setminus\pi_{1}[\text{int }G(f)]$, let $g(x)=f(x)$. For each $n$, we construct $G(g)$ on $\overline{O}_{n}$ to contain any of $(x_{0},x_{1})$, $(x_{1},x_{2})$,\dots,$ (x_{p-1},x_{0})$ for which $x_{i}\in O_{n}$, and some $(a_{n},\max f(\overline{O}_{n}))$, $(b_{n},\min f(\overline{O}_{n}))$ where $a_{n},b_{n}\in\overline{O}_{n}$. To that end, let $C=\{x_{0},x_{1},\dots,x_{p-1}\}\cup\bigcup_{n\in\mathbb{N}}\{a_{n},b_{n}\}$. Define $g(x)$ to be $f(x)$ if $x\in C$.

 Note that for each $n$, $C\cap O_{n}$ is finite. Define $g$ on $O_{n}\setminus C$ to be single-valued and continuous according to the following conditions:
\begin{enumerate}
\item $g(x)\subseteq f(x)$,
\item $g$ is light on $O_{n}\setminus C$ and
\item if $x$ is in $C$ or $\text{bd }O_{n}$, then for any component $U$ of $O_{n}\setminus C$ with $x\in\overline{U}$, $\overline{G(g|_{U})}\cap (\{x\}\times[0,1])=\{x\}\times g(x)$.
\end{enumerate}
That $g$ may be light on $O_{n}\setminus C$ while maintaining $G(g)\subseteq G(f)$ follows from the fact that $O_{n}\subseteq \pi_{1}[\text{int }G(f)]$. Regarding (3), since $C\cap O_{n}$ is finite and $f$ is weakly continuous, $g$ may also be constructed such that as $y$ approaches $x$ from within $U$, the graph of $g$ is a ray with remainder $g(x)$. Therefore such a map $g$ exists. Note that by (1) and the fact that $g(x)=f(x)$ on $C$, $g[\overline{O}_{n}]=f[\overline{O}_{n}]$.

Note that $(x_{0},x_{1},\dots,x_{p-1})$ is a periodic cycle of $g$. By this construction, $g$ is light and almost nonfissile on each $O_{n}$ and $G(g)\cap V_{O_{n}}$ is connected. Note that if $x\in \text{bd } O_{n}$, condition (3) becomes $\overline{G(g|_{U})}\cap (\{x\}\times[0,1])=\{x\}\times g(x)=\{x\}\times f(x)$. Then since $g|_{[0,1]\setminus \pi_{1}[\text{int }G(f)]}=f|_{[0,1]\setminus \pi_{1}[\text{int }G(f)]}$, $g$ is almost nonfissile and light on $[0,1]$, and $G(g)$ is connected. It remains to show $g$ has the intermediate value property. Since $g(x)$ is connected for each $x\in[0,1]$, it is sufficient to show that $g$ is weakly continuous.

We show that $g$ is weakly continuous from the left. The proof that $g$ is weakly continuous from the right is similar. Let $(x,y)\in G(g)$ with $x>0$. Suppose first $x\in O_{n}$ for some $n$. If $x\in C$, then by (3) there is a sequence $\{(x_{i},y_{i})\}_{i\in\omega}$ in $G(g)$ such that $x_{i}\in O_{n}\cap (0,x)$ for all $i$ and $(x_{i},y_{i})\rightarrow (x,y)$. Thus $g$ is weakly continuous at $x$ from the left. If $x\notin C$, then since $g$ is single-valued and continuous on $O_{n}\setminus C$, it follows that $g$ is weakly continuous at $x$ from the left.

Next suppose $x\not\in O_{n}$ for any $n$. Then either $x\in [0,1]\setminus \overline{\pi_{1}[\text{int }G(f)]}$, $x= \sup O_{n}$ for some $n$, or there is a subsequence $O_{n_{k}}$ such that $x>\sup O_{n_{k}}$ for all $k$ but $x=\sup\bigcup_{k}O_{n_{k}}$.

Case 1: Suppose $x\in [0,1]\setminus \overline{\pi_{1}[\text{int }G(f)]}$. Since $f$ is weakly continuous, there is a sequence $\{(x_{i},y_{i})\}_{i\in\omega}$ in $G(f)$ such that $x_{i}<x$ for all $i$ and $(x_{i},y_{i})\rightarrow (x,y)$. Then there is some $N\in\mathbb{N}$ such that for $i\geq N$, $x_{i}\in[0,1]\setminus\pi_{1}[\text{int }G(f)]$. Since $g$ agrees with $f$ on $[0,1]\setminus\overline{\pi_{1}[\text{int }G(f)]}$, $\{(x_{i},y_{i})\}_{i\geq N}$ is a sequence in $G(g)$ converging to $(x,y)$.

Case 2: Suppose $x=\sup O_{n}$ for some $n$. Then by (2), there is a sequence $\{(x_{i},y_{i})\}_{i\in\omega}$ in $G(g)$ such that $x_{i}\in O_{n}$ for all $i$ and $(x_{i},y_{i})\rightarrow (x,y)$.

Case 3: Suppose there is a sequence $\{O_{n_{k}}\}_{k\in\omega}$ such that $\sup O_{n_{k}}<x$ and $x=\sup\bigcup_{k}O_{n_{k}}$. Note that any such sequence may be ordered so that $O_{n_{k}}=(c_{k},d_{k})$ where $d_{k}<c_{k+1}$, $c_{k}\rightarrow x$, and $d_{k}\rightarrow x$. Then $d_{k}-c_{k}\rightarrow 0$, i.e. $\text{diam }O_{n_{k}}\rightarrow 0$. Let $\{(x_{i},y_{i})\}_{i\in\omega}$ be a sequence in $G(f)$ such that $x_{i}<x$ for all $i$ and $(x_{i},y_{i})\rightarrow (x,y)$. Recall that $f$ Define a sequence $\{(x'_{i},y_{i})\}_{i\in\omega}$ in $G(g)$ where $x'_{i}$ is a point of some $O_{n_{i}}$ with $y_{i}\in g(x'_{i})$ if $x_{i}\in O_{n_{i}}$ and $x'_{i}=x_{i}$ if $x_{i}\in[0,1]\setminus\pi_{1}[\text{int }G(f)]$. Note $d((x'_{i},y_{i}),(x_{i},y_{i}))=|x'_{i}-x_{i}|<\text{diam }O_{n_{i}}$ if $x_{i}\in O_{n_{i}}$. Let $\epsilon>0$ and $N_{1}$ such that if $i\geq N_{1}$, $d((x_{i},y_{i}),(x,y))<\frac{\epsilon}{2}$. Since $\text{diam }O_{n_{i}}\rightarrow 0$, there is some $N_{2}$ such that if $i\geq N_{2}$, then $\text{diam }O_{n_{i}}<\frac{\epsilon}{2}$. Then for $i\geq \max\{N_{1},N_{2}\}$,
\[
d((x'_{i},y_{i}),(x,y))\leq d((x',y_{i}),(x_{i},y_{i}))+d((x_{i},y_{i}),(x,y))<\frac{\epsilon}{2}+\frac{\epsilon}{2}=\epsilon.
\]
Then $(x'_{i},y_{i})\rightarrow (x,y)$. Therefore $g$ is weakly continuous from the left. By a similar argument, $g$ is weakly continuous from the right. Thus $g$ is weakly continuous. Then $g$ has the intermediate value property.

%
%
%
\end{proof}

\subsection{Indecomposability giving rise to periodicity}  \label{d.2}

\begin{lemma}\label{irred1}
Let $f:[0,1]\rightarrow 2^{[0,1]}$ be such that $f$ is upper semicontinuous, surjective, and has the intermediate value property. Further suppose that $\invlim\{[0,1],f\}$ is irreducible between $x$ and $y$. For $k\geq 0$, let $J_{k} = \overline{\bigcup_{n\geq k}f^{n-k}(\overline{x_{n}y_{n}})}$. Then for each $k$, $J_{k}$ is a closed subinterval of $[0,1]$ with $f(J_{k+1})=J_{k}$.
\end{lemma}

\begin{proof}
Since $f$ has intermediate value property, $x_{i}\in f(x_{i+1})$, and $y_{i}\in f(y_{i+1})$ for all $i$, $\overline{x_{i}y_{i}}\subseteq f(\overline{x_{i+1}y_{i+1}})$. Thus if $n_{2}>n_{1}$, $f^{n_{1}}(\overline{x_{n_{1}}y_{n_{1}}})\subseteq f^{n_{2}}(\overline{x_{n_{2}}y_{n_{2}}})$. So \[\overline{x_{0}y_{0}}\subseteq f(\overline{x_{1}y_{1}})\subseteq f^{2}(\overline{x_{2}y_{2}})\subseteq \dots.\]
Since $f$ has the intermediate value property, for each $k$ $f^{k}(\overline{x_{n}y_{n}})$ is an interval. Thus $J_{k}$ is a closed subinterval of $[0,1]$.
Note for $n\geq k+1$, \[f(J_{k+1})\supseteq f(f^{n-(k+1)}(\overline{x_{n}y_{n}}))=f^{n-k}(\overline{x_{n}y_{n}}).\] So $f(J_{k+1})\supseteq \bigcup_{n\geq k+1} f^{n-k}(\overline{x_{n}y_{n}})$. But because $\overline{x_{k}y_{k}}\subseteq f(\overline{x_{k+1}y_{k+1}})$, we have \[f(J_{k+1})\supseteq \bigcup_{n\geq k} f^{n-k}(\overline{x_{n}y_{n}}).\]
As $f(J_{k+1})$ is closed, \[f(J_{k+1})\supseteq \overline{\bigcup_{n\geq k} f^{n-k}(\overline{x_{n}y_{n}})} =J_{k}.\]

Similarly for $n\geq k+1$, \[J_{k}\supseteq f^{n-k}(\overline{x_{n}y_{n}})= f(f^{n-(k+1)}(\overline{x_{n}y_{n}})).\]
Thus $J_{k}\supseteq f\left(\bigcup_{n\geq k+1}f^{n-(k+1)}(\overline{x_{n}y_{n}})\right)$. Since $J_{k}$ is closed and $f$ has the intermediate value property and is therefore weakly continuous, by Theorem 3.12 of \cite{Dunn}, \[\overline{f\left(\bigcup_{n\geq k+1}f^{n-(k+1)}(\overline{x_{n}y_{n}})\right)}=f\left(\overline{\bigcup_{n\geq k+1}f^{n-(k+1)}(\overline{x_{n}y_{n}})}\right).\]
Thus \[J_{k}=\overline{f\left(\bigcup_{n\geq k+1}f^{n-(k+1)}(\overline{x_{n}y_{n}})\right)}=f\left(\overline{\bigcup_{n\geq k+1}f^{n-(k+1)}(\overline{x_{n}y_{n}})}\right)=f(J_{k+1}).\]

\end{proof}

\begin{lemma}\label{irred}
Let $f:[0,1]\rightarrow 2^{[0,1]}$ be such that $G(f)$ is connected and $f$ is upper semicontinuous, surjective, and has the intermediate value property. Further suppose that $\invlim\{[0,1],f\}$ is irreducible between $x$ and $y$. If $0<c<d<1$, then there is some $N\in\omega$ such that $n>N$ implies $[c,d]\subseteq f^{n}([x_{n},y_{n}])$.
\end{lemma}

\begin{proof}
Let $J_{k} = \overline{\bigcup_{n\geq k}f^{n-k}(\overline{x_{n}y_{n}})}$, as in Lemma \ref{irred1}, and let $J = \invlim\{J_{k},f|_{J_{k+1}}\}$. As $f|_{J_{k+1}}$ also has the intermediate value property, is surjective, is upper semicontinuous, and $G(f|_{J_{k}})$ is connected by Theorem 3.12 of \cite{Dunn}, $J$ is a subcontinuum of $\invlim\{[0,1],f\}$. Note $x,y\in J$; hence $J=\invlim\{[0,1],f\}$. As $f$ is surjective, $J_{0}=[0,1]$ and $[0,1] = \overline{\bigcup_{n\geq 0}f^{n}(\overline{x_{n}y_{n}})}$. Then because $f^{n}(\overline{x_{n}y_{n}})\subseteq f^{n+1}(\overline{x_{n+1}y_{n+1}})$ for each $n$, there is some $N\in\omega$ such that if $n\geq N$, $[c,d]\subseteq f^{n}(\overline{x_{n}y_{n}})$.
\end{proof}

\begin{lemma}
Suppose $f:[0,1]\rightarrow 2^{[0,1]}$ is upper semicontinuous, surjective, has the intermediate value property. If there are $p,q\in(0,1)$ and $r,s\in\omega$ with $0\in f^{r}(p)$ and $1\in f^{s}(q)$, then $f$ is organic.
\end{lemma}

\begin{proof}
Suppose $\invlim\{[0,1],f\}$ is irreducible between $x$ and $y$. Then by Lemma \ref{irred}, there are positive integers $N_{r}$ and $N_{s}$ such that if $n>N_{r}$, $p\in f^{n-r}(\overline{x_{n}y_{n}})$ and if $n>N_{s}$, $q\in f^{n-s}(\overline{x_{n}y_{n}})$. So if $n>N_{r}+N_{s}$, $f^{n}(\overline{x_{n}y_{n}})=[0,1]$.
\end{proof}

\begin{theorem}\label{IndecomposablePeriodic}
If $f:[0,1]\rightarrow 2^{[0,1]}$ is upper semicontinuous, organic and has the intermediate value property and $\invlim\{[0,1],f\}$ is indecomposable, then $f$ has a periodic cycle with a period that is not a power of 2.
\end{theorem}

\begin{proof}
Since $\invlim\{[0,1],f\}$ is indecomposable, there are three points $x$, $y$, and $z$ such that $\invlim\{[0,1],f\}$ is irreducible between any two of them. Because $f$ is organic, there exists some $n$ such that $f^{n}(\overline{x_{n}y_{n}})=f^{n}(\overline{y_{n}z_{n}})=f^{n}(\overline{x_{n}z_{n}})=[0,1]$. Without loss of generality, suppose $x_{n}<y_{n}<z_{n}$. As $f$ is upper semicontinuous and has the intermediate value property, $f^{n}(y_{n})$ is a closed interval. Thus either $y_{n}\in \text{int}(f^{n}(y_{n}))$, $f^{n}(y_{n})\subseteq [0,y_{n}]$, or $f^{n}(y_{n})\subseteq [y_{n},1]$.

\emph{Case 1:} Suppose $y_{n}\in \text{int}(f^{n}(y_{n}))$. Then there are numbers $c$ and $d$ such that $c<y_{n}<d$ and $f(y_{n})=[c,d]$. As $f$ has the intermediate value property, $f$ is weakly continuous. Thus, there exist sequences $\{(a_{i},c_{i})\}_{i\in\omega}$ and $\{(b_{i},d_{i})\}_{i\in\omega}$ in $G(f^{n})$ such that for all $i$ $a_{i},b_{i}<y_{n}$, $a_{i},b_{i}\rightarrow y_{n}$, $c_{i}\rightarrow c$, and $d_{i}\rightarrow d$. Furthermore these sequences may be chosen such that $a_{i}\leq b_{i}\leq a_{i+1}$ for all $i$. Then because $c<y_{n}<d$, there is some $N\in\mathbb{N}$ such that $i\geq N$ implies $c_{i}<y_{n}<d_{i}$. Since $f^{n}$ has the intermediate value property, for $i\geq N$, there is a point $p_{i}\in[a_{i},b_{i}]$ with $y_{n}\in f^{n}(p_{i})$. Furthermore $p_{i}\rightarrow y_{n}$ since $a_{i},b_{i}\rightarrow y_{n}$.

Note that as $y_{n}\in \text{int}f^{n}(y_{n})$ and $p_{i}\rightarrow y_{n}$, $p_{i}\in f^{n}(y_{n})$ for cofinitely many $i$. Then $y_{n}\in f^{n}(p_{i})$ and $p_{i}\in f^{n}(y_{n})$ for cofinitely many $i$. Thus, for any $k\in\mathbb{N}$, there is a periodic orbit of the form $(q_{0},\dots, q_{2kn})$ where for $j=0\dots,k$, $q_{2jn}=y_{n}$ and for $j=0,\dots, k-1$, $q_{(2j+1)n}$ is a distinct member of the $p_{i}$'s. In particular, $k=3$ gives a periodic cycle with a period that is not a power of 2, satisfying the conclusion of the theorem.

\emph{Case 2:} Suppose $f^{n}(y_{n})\subseteq [0,y_{n}]$. Then either $f^{n}(y_{n})=\{y_{n}\}$ or there is some value $b\in f^{n}(y_{n})$ with $b<y_{n}$. If $f^{n}(y_{n})=y_{n}$, then there are values $a,b\in[x_{n},y_{n})$ such that $0\in f^{n}(a)$ and $1\in f^{n}(b)$. Thus there is a closed interval $J_{1}\subseteq \overline{ab}\subseteq [x_{n},y_{n})$ and a restriction $f^{n}|_{J_{1}}^{[y_{n},z_{n}]}$ of $f^{n}$ such that $f^{n}|_{J_{1}}^{[y_{n},z_{n}]}(J_{1})=[y_{n},z_{n}]$ \cite{OteyRyden}.

We show that such a $J_{1}$ also exists if there is some $b\in f^{n}(y_{n})$ with $b<y_{n}$. By the weak continuity of $f^{n}$ there is a sequence $\{(a_{i},b_{i})\}_{i\in\omega}$ such that for all $i$, $x_{n}\leq a_{i}<y_{n}$, $b_{i}\in f^{n}(a_{i})$, $a_{i}\rightarrow y_{n}$, and $b_{i}\rightarrow b$. Thus there is some $b_{N}<y_{n}$. Let $q\in [x_{n},y_{n})$ be a point such that $1\in f^{n}(q)$. Then $f^{n}(\overline{b_{N}q})\supseteq[b,1]\supseteq[y_{n},z_{n}]$, so there is a closed interval $J_{1}\subseteq \overline{b_{N}q}\subseteq [x_{n},y_{n})$ and a restriction $f^{n}|_{J_{1}}^{[y_{n},z_{n}]}$ of $f^{n}$ such that $f^{n}|_{J_{1}}^{[y_{n},z_{n}]}(J_{1})=[y_{n},z_{n}]$.


As $f^{n}([x_{n},y_{n}])\supseteq J_{1}$, there is a closed subinterval $J_{2}$ of $[x_{n},y_{n}]$ and a restriction $f^{n}|_{J_{2}}^{J_{1}}$ of $f^{n}$ such that $f^{n}|_{J_{2}}^{J_{1}}(J_{2})=J_{1}$. Similarly there is a closed subinterval $J_{3}$ of $[y_{n},z_{n}]$ and a restriction $f^{n}|_{J_{3}}^{J_{2}}$ of $f^{n}$ such that $f^{n}|_{J_{3}}^{J_{2}}(J_{3})=J_{2}$.

Thus $J_{3}\subseteq f^{n}|_{J_{1}}^{[y_{n},z_{n}]}(f^{n}|_{J_{2}}^{J_{1}}(f^{n}|_{J_{3}}^{J_{2}}(J_{3})))\subseteq f^{3n}(J_{3})$. Then there is a periodic orbit $(q_{0},\dots, q_{3n})$ with $q_{0}=q_{3n}=q\in J_{3}$, $q_{n}\in J_{2}$, and $q_{2n}\in J_{1}$. Suppose $q_{2n}=q$. Then $q\in J_{1}\cap J_{3}\subseteq [x_{n},y_{n}]\cap[y_{n},z_{n}]=\{y_{n}\}$. But then we would have $y_{n}\in J_{1}$, a contradiction. So $q\neq q_{2n}$.

Let $s$ be the period of $(q_{0},\dots, q_{3n})$. Then $s\mid 3n$. As $q_{2n}\neq q_{0}=q$, $s\nmid 2n$. If $s\mid n$, then $s\mid 2n$, a contradiction. It follows that $s\nmid n$. Therefore $3\mid s$, and $s$ is not a power of 2 as desired. The case for $f^{n}(y_{n})\subseteq[y_{n},1]$ follows from a similar argument.
\end{proof}

%
%

\end{document}